\documentclass[]{article}
\usepackage{amsmath}
\usepackage{amsfonts}
\usepackage{amssymb}
\usepackage[english]{babel}
\usepackage{graphicx}
\usepackage{amsthm}
\theoremstyle{plain}

\newtheorem*{theorem*}{Theorem}
\newtheorem{theorem}{Theorem}
\theoremstyle{definition}
\newtheorem{definition}{Definition}[section]
\theoremstyle{lemma}
\newtheorem{lemma}{Lemma}[section]
\theoremstyle{corollary}
\newtheorem{corollary}{Corollary}[section]
\theoremstyle{proposition}
\newtheorem{proposition}{Proposition}[section]

\theoremstyle{remark}

\title{Outer Billiards with Contraction: Attracting Cantor Sets}
\author{In-Jee Jeong}

\begin{document}

\global\long\def\l{\lambda}
\global\long\def\ep{\epsilon}

\maketitle

\begin{abstract}
	We consider the outer billiards map with contraction outside polygons.
	We construct a 1-parameter family of systems such that each system
	has an open set in which the dynamics is reduced to that of a piecewise
	contraction on the interval. Using the theory of rotation numbers,
	we deduce that every point inside the open set is asymptotic to either
	a single periodic orbit (rational case) or a Cantor set (irrational
	case). In particular, we deduce existence of an attracting Cantor
	set for certain parameter values. Moreover, for a different choice
	of a 1-parameter family, we prove that the system is uniquely ergodic;
	in particular, the entire domain is asymptotic to a single attractor.
\end{abstract}

\section{Introduction}
Let $P$ be a convex polygon in the plane and $0<\l<1$ be a real
number. Given a pair $(P,\l)$, we define the outer billiards with
contraction map $T$ as follows. For a generic point $x\in\mathbb{R}^{2}\backslash P$,
we can find a unique vertex $v$ of $P$ such that $P$ lies on the
left of the ray starting from $x$ and passing through $v$. Then
on this ray, we pick a point $y$ that lies on the other side of $x$
with respect to $v$ and satisfies $|xv|:|vy|=1:\l$. We define $Tx=y$
(Figure \ref{fig:The-outer-billiard}). The map $T$ is well-defined
for all points on $\mathbb{R}^{2}\backslash P$ except for points
on the union of singular rays extending the sides of $P$. 

\begin{figure}[h]
	\centering
	\includegraphics{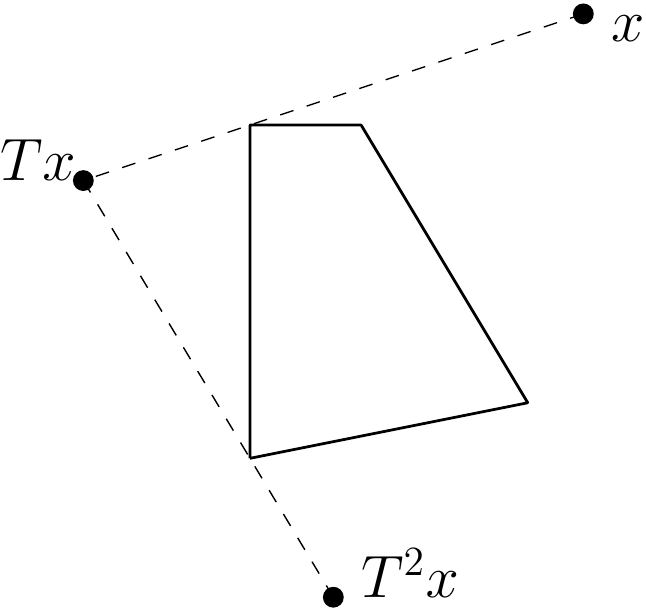}
	\caption{\label{fig:The-outer-billiard}Outer billiards with contraction}
\end{figure}

The case $\l=1$ corresponds to the usual outer billiards map. This
map has been studied by many people \cite{NicolasBedaride2011,MR2854095,MR2562898,MR1354670,MR1145593,1006.2782}; \cite{Tabachnikov1995} is a classical reference. We first point out
that while it is highly nontrivial to show that the outer billiards
map can have unbounded orbits for certain polygons \cite{MR2318496},
it is straightforward to see that with any contraction $0<\l<1$,
every orbit is bounded. 
While the system itself has several interesting behavior, we state a few motivations to study outer billiards with contraction. 
\begin{itemize}
	\item{Electrical engineering: It is known that the outer billiards map is connected to certain systems from electrical engineering \cite{MR1201496,MR2249414,MR944132}. Then, applying contraction may model more realistic systems in which dissipation of energy occurs; see \cite{MR2675099}.}
	\item{Application to the outer billiards map: By considering the limit $\l \rightarrow 1$, one gains some information regarding the structure of the set of periodic points of the outer billiards map. See \cite{Jeong2013} for some illustrations of this phenomenon.}
	\item{An analogy of inner billiards with contraction: Contractive inner billiards have been studied, for example, in  \cite{GianluigiDelMagno2012,GianluigiDelMagno2013a}. Then it is natural to consider the outer billiards counterpart; we have learned this motivation from \cite{GianluigiDelMagno2013}.}
\end{itemize}
Recent papers \cite{MR2465664,Nogueira2012} discuss a few more motivations
to study piecewise contractions in general. 

Experiment suggests that when we pick $P$ and $\l$ at random, (1)
there are only finitely many periodic orbits and (2) all orbits are
asymptotically periodic. Indeed, these two phenomena are expected
from a generic piecewise contracting map; \cite{MR2465664,Nogueira2012,ArnaldoNogueira2013}
prove exactly results of this kind. In \cite{MR2465664}, authors
consider a large class of piecewise affine contractions on the plane
and show that for almost every choice of parameters, any orbit is
attracted to a periodic one. In \cite{Nogueira2012,ArnaldoNogueira2013},
authors consider piecewise contractions on the interval and show that
the number of periodic orbits is always finite, and asymptotic periodicity
for almost every choice of parameters by adding a smoothness assumption
on each piece. We point out special cases where regularity statements
(1) and (2) above can be explicitly proved: When $P$ is either a
triangle, parallelogram, or a regular hexagon, then for each $0<\l<1$,
there exists only finitely many periodic orbits for $T$ and all orbits
are asymptotically periodic. Proofs of this fact are carried out in
\cite{Jeong2013,GianluigiDelMagno2013}, using different methods.

In a sense, our goal will be the opposite to that of aforementioned
papers. The main result of this paper is the following.

\begin{theorem*}
There exists uncountably many pairs $(P,\l)$ for which $T$ has a	Cantor set as an attractor. 
\end{theorem*}

It will turn out that the set of parameter values where we can show
existence of such Cantor sets also have zero measure in the parameter
space. As far as we know, it is the first explicit description of
a 2-dimensional system (that is not a direct product of two 1-dimensional
systems) having an attracting Cantor set. An outline is as follows:
\begin{enumerate}
	\item We find 1-parameter family of pairs $(P,\l)$ which has a thin forward
	invariant rectangle. (Definition \ref{def:inv_rectangle}, Lemma \ref{lem:invariance})
	\item The thin rectangle is divided into two rectangles, and the dynamics
	on each rectangle is reduced to a one-dimensional piecewise affine
	contraction. (Definition \ref{def:1-dimensional_maps}, Lemma \ref{lem:limit_set_reduction})
	\item We show that the rotation numbers of these maps are well-defined and
	vary continuously with the parameter. (Lemma \ref{lem:rotation number}) 
	\item If this rotation number is irrational, there must be an attracting
	Cantor set. (Theorem \ref{thm:two_cases})
	\item Finally, we conclude the proof by show that the rotation number is
	non-constant as a function of the parameter by computing 2 examples.
	(Theorem \ref{thm:cantorset})
\end{enumerate}
A brief theory of rotation numbers for discontinuous circle maps is
given in Section \ref{sec:Rotation-Theory-for}, and the above outline
is carried out in Section \ref{sec:Attracting-Cantor-Sets}. In Section
\ref{sec:Discussions}, we consider more general parameter values
and discuss global behavior of the system. In particular, we
show that there are cases where the attracting Cantor set is the global
attractor.

\section{Rotation Theory for Discontinuous Circle Maps \label{sec:Rotation-Theory-for}}

Recall that the rotation number for a circle homeomorphism $f:S^{1}\rightarrow S^{1}$
is defined by 
\begin{equation}
\rho(F)=\lim_{n\rightarrow\infty}\frac{F^{n}(x)-x}{n}\label{eq:rotation number}
\end{equation}
where $F$ is any lift of $f$ into a homeomorphism of $\mathbb{R}$
and $x$ is an arbitrary real number. Once we fix $F$, this limit
exists and independent on $x$. If we have two lifts $F_{1}$ and
$F_{2}$, then $\rho(F_{1})-\rho(F_{2})$ is an integer so that the
rotation number of $f$, $\rho(f)$ is uniquely determined mod 1.
Then, $\rho(f)$ is rational if and only if $f$ has a periodic point.
On the other hand, when $\rho(f)$ is irrational, the $\omega$-limit
set of a point (which is independent on the point) is either the whole
circle or a Cantor set. 

Rhodes and Thompson, in \cite{MR856518,MR1099095}, develops a theory
of rotation number for a large class of functions $f:S^{1}\rightarrow S^{1}$.
We only state the results that we need. A map (not necessarily continuous)
$f:S^{1}\rightarrow S^{1}$ is in class $\mathcal{S}$ if and only
if it has a lift $F:\mathbb{R}\rightarrow\mathbb{R}$ such that $F$
is strictly increasing and $F(x+1)=F(x)+1$ for all $x\in\mathbb{R}$.
Now given such a lift $F$, since it is strictly increasing, we can
define $F^{-}$ and $F^{+}$, where they are continuous everywhere
from the left and from the right, respectively and coincide with $F$
whenever $F$ is continuous. Then we consider the filled-graph of
$F$ defined by 
\[
\Gamma(F):=\{(x,y)|0\leq x\leq1,\, F^{-}(x)\leq y\leq F^{+}(x)\}
\]
which is simply the graph of $F$ where all the jumps are filled with
vertical line segments. We have restricted the set to the region $0\leq x\leq1$
to make it compact. We will consider the Hausdorff metric on the collection
of $\Gamma(F)$ where $F$ is some lift of $f\in\mathcal{S}$. 

In terms of this theory, it does not matter how maps are defined at
points of discontinuity. Therefore, we follow the convention of \cite{Brette2003};
for maps in $\mathcal{S}$, a periodic point (or periodic orbit) of
$f$ will mean a periodic point (or periodic orbit) of a map which
might differ from $f$ at some points of discontinuity. With this
convention, we have the following theorems. 
\begin{theorem*}
	\cite{MR856518} The rotation number $\rho(f)$ is well-defined for
	$f\in\mathcal{S}$ up to mod 1 by the equation \ref{eq:rotation number}
	where $F$ is any strictly increasing degree 1 lift of $f$. This
	number does not change if we redefine $f$ at its points of discontinuity.
	Moreover, $\rho(f)$ is rational if and only if $f$ has a periodic
	point. 
	\begin{theorem*}
		\cite{MR1099095} Let $F_{\theta}$ be a family of strictly increasing
		degree 1 functions $\mathbb{R}\rightarrow\mathbb{R}$ for $\theta\in[0,1]$.
		If $\Gamma(F_{\theta})\rightarrow\Gamma(F_{0})$ as $\theta\rightarrow0$
		in the Hausdorff topology, then $\rho(F_{\theta})\rightarrow\rho(F_{0})$. 
	\end{theorem*}
\end{theorem*}
Finally, regarding the $\omega$-limit set we refer to \cite{Brette2003}.
\begin{theorem*}
	\cite{Brette2003} If $f\in\mathcal{S}$ has a rational rotation number
	$p/q$, $\omega(x)$ gives a $q$-periodic orbit of $f$ for all $x\in S^{1}$.
	If $f$ has an irrational rotation number, $\omega(x)=\omega(y)$
	for all $x,y\in S^{1}$ and it is either $S^{1}$ or homeomorphic
	to a Cantor set.
\end{theorem*}

\section{Attracting Cantor Sets\label{sec:Attracting-Cantor-Sets}}

We begin with some general discussion of the system. When $P$ is
a convex $n$-gon, with clockwisely oriented vertices $A_{1},...,A_{n}$,
then let $S$ denote the union of $n$ singular rays $\overrightarrow{A_{1}A_{2}},...,\overrightarrow{A_{n-1}A_{n}},\overrightarrow{A_{n}A_{1}}$
where $T$ is not well-defined. Then map $T$ is well-defined indefinitely
on the set $X:=\mathbb{R}^{2}\backslash\big(P\cup(\cup_{i=0}^{\infty}T_{\l}^{-i}S)\big)$,
which has full measure on $\mathbb{R}^{2}\backslash P$. We say that
an ordered set of points $(w_{1},...,w_{n})$ is a \textit{periodic
	orbit} if for each $1\leq i\leq n$, either $w_{i}\in X$ and $Tw_{i}=w_{i+1}$
(where $w_{n+1}:=w_{1})$ or $w_{i}\in S$ and $w_{i+1}$ is one of
(at most) two natural choices for $Tw_{i}$ which would make $T$
continuous from one side of the plane. Note that this generalized
notion of a periodic orbit is consistent with the one described in
the previous section. 

Now we proceed to the description of our 1-parameter family of systems,
with parameter $a$ varying in the closed interval $[0.3,0.6]$. The
polygon $P$ will be the quadrilateral with vertices $A=(0,1)$, $B=(a,1)$,
$C=(1,0.2)$, and $D=(0,0)$. The contraction factor $\l$ will be
equal to 0.8, independent on $a$. For simplicity, we will write $\l$
instead of $0.8$. We also set $\mu=\l^{-1}$ 

Let us describe how the points in Figure \ref{fig:Construction_of_Points}
are constructed. To begin with, the point $F$ (respectively, $G$)
lies on the line $\overleftrightarrow{AD}$ (resp. $\overleftrightarrow{BC}$)
and satisfies $|\overline{AD}|:|\overline{DF}|=\l$ (resp. $|\overline{BC}|:|\overline{CG}|=1:\mu$).
Then we see that the line $\overleftrightarrow{FG}$ is parallel to
the line $\overleftrightarrow{AB}$. Then the point $E$ (resp. $H$)
is the intersection of the line $\overleftrightarrow{AB}$ with $\overleftrightarrow{DG}$
(resp. $\overleftrightarrow{AB}$ with $\overleftrightarrow{FC}$).
Finally, we pick the point $I'$ on the segment $\overline{BH}$ which
satisfies $|\overline{AB}|:|\overline{BI'}|=\l:1$. Then we define
$I$ to be the point on the segment $\overline{EA}$ such that $T^{2}I=I'$.

\begin{figure}[h]
	\centering
	\includegraphics{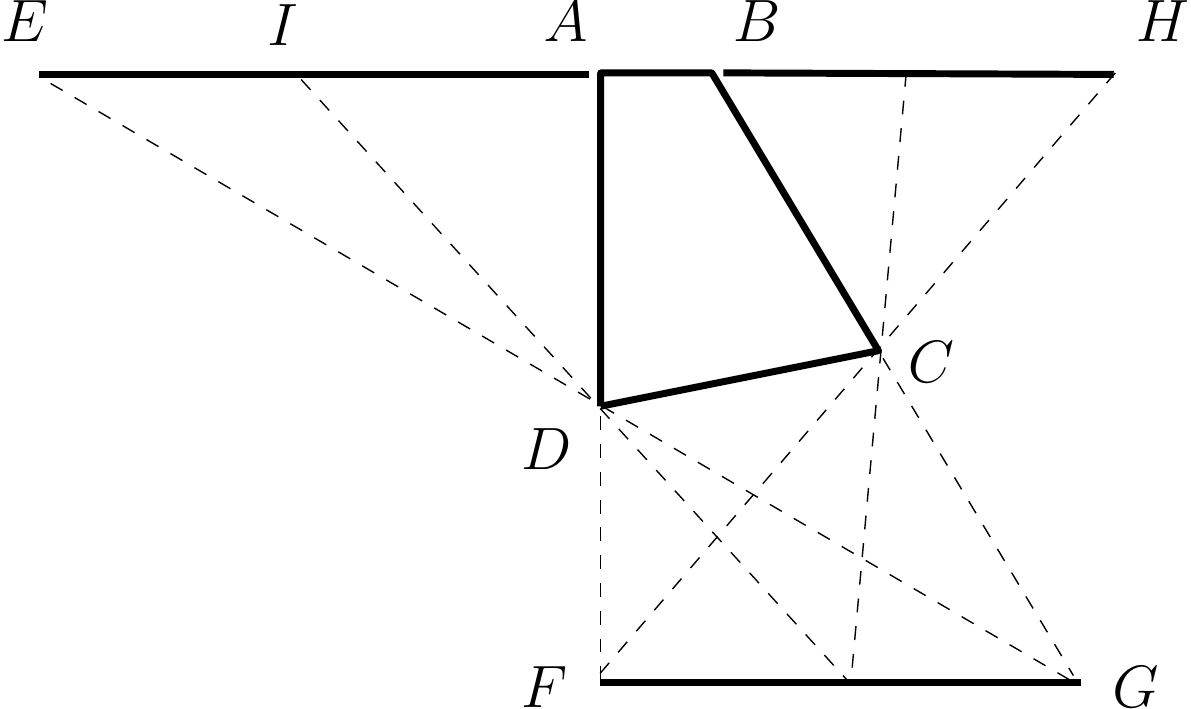}
	\caption{\label{fig:Construction_of_Points}Constructing extra points}
\end{figure}

We say that a region $Y\subset\mathbb{R}^{2}\backslash P$ is \textit{forward invariant }if for each $x\in Y\cap X$, there exists $n>0$ such that
$T^{n}x\in Y$. 
\begin{definition}[The Invariant Rectangles]
	\label{def:inv_rectangle} Let $(-l,1)$ be the coordinates for the
	point $E$. For sufficiently small $\ep>0$, we define $R_{1}$ as
	the closed filled-in rectangular region with vertices $(0,1)$, $(0,1+\ep^{2})$,
	$(-l+\ep,1)$, and $(-l+\ep,1+\ep^{2})$. Then reflect $R_{1}$ over
	$\overline{EA}$ to obtain $R_{2}$. \end{definition}
\begin{lemma}[Invariance]
	\label{lem:invariance} For sufficiently small $\ep>0$, 
	\begin{enumerate}
		\item any point in $R_{1}$ has a forward interate in $R_{2}$ and vice
		versa. In particular, $R_{1}$ and $R_{2}$ are forward invariant
		regions, and
		\item Set $\hat{f}$ and $\hat{g}$ as the first-return maps to regions
		$R_{1}$ and $R_{2}$, respectively Then these maps preserve the vertical
		partition of the rectangles; that is, if two points in $R_{1}$ (or
		in $R_{2}$) have the same $x$-coordinates, there forward iterates
		also have the same $x$-coordinates. Moreover, given a point in $R_{1}$
		(resp. $R_{2}$), the sequence of $y$-coordinates formed by its $\hat{f}$-iterates
		($\hat{g}$-iterates, resp.) converges to $1$, which is the $y$-coordinate
		of $\overline{EA}$.
	\end{enumerate}
\end{lemma}
The dynamics is not well-defined at certain points on $R_{1}\cup R_{2}$.
For example, the segment obtained as the intersection of those two
rectangles is singular. We always follow the convention that the dynamics
at such points is defined to be the one that makes $T$ continuous
in the region we are looking at. That is, the segment $\overline{BH}$,
viewed as a subset of $T^{2}(R_{1})$, will reflect on the vertex
$A$, but viewed as a subset of $T^{2}(R_{2})$, it will reflect on
the vertex $B$.
\begin{proof}
	The line passing through the point $I$ and perpendicular to $\overline{EA}$
	divides $R_{2}$ into two closed rectangles $R_{2l}$ (on the left)
	and $R_{2r}$ (on the right). The first claim is proved with the following
	three crucial containments: (see Figures \ref{fig:Iterates} and \ref{fig:Iterates-1})
	\begin{itemize}
		\item $T^{3}R_{1}\subset R_{2}$
		\item $T^{4}R_{2l}\subset R_{2r}$
		\item $T^{3}R_{2r}\subset R_{1}$ (which together with $T^{3}R_{1}\subset R_{2}$
		implies $T^{6}R_{2r}\subset R_{2}$)
	\end{itemize}

Let us prove above containments. First we note that for $\ep>0$ small,
every points on the image $T(R_{1}\cup R_{2})$ lies below the line
$\overleftrightarrow{BC}$ so that entire rectangle $T(R_{1}\cup R_{2})$
reflects on the vertex $C$ and not $B$ (this is the only place we
need $\ep$ to be small). Hence $T^{2}(R_{1}\cup R_{2})$ is another
rectangle intersecting $\overline{BH}$. Now, to prove $T^{3}R_{1}\subset R_{2}$,
it is enough to check that the $x$-coordinate of the point $H$,
when reflected and contracted by $\l$ at the vertex $A$, is contained
in the segment $R_{1}\cap R_{2}$. This holds for all $0.3\leq a\leq0.6$.
Next, it is obvious from $T^{3}R_{1}\subset R_{2}$ that $T^{3}R_{2r}\subset R_{1}$.
Lastly, it is another calculation to show that $T^{4}R_{2l}\subset R_{2r}$
for all $0.3\leq a\leq0.6$.

Next, since each iterate of $T$ contracts the distance between the
point and either the line $\overleftrightarrow{AB}$ or $\overleftrightarrow{FG}$
by $\l$ and it takes for points in $R_{1}\cup R_{2}$ at least $T^{3}$
to return to $R_{1}\cup R_{2}$, the convergence towards the segment
$\overline{EA}$ is exponential.
\end{proof}

\begin{figure}[h]
	\centering
	\includegraphics[scale=0.3]{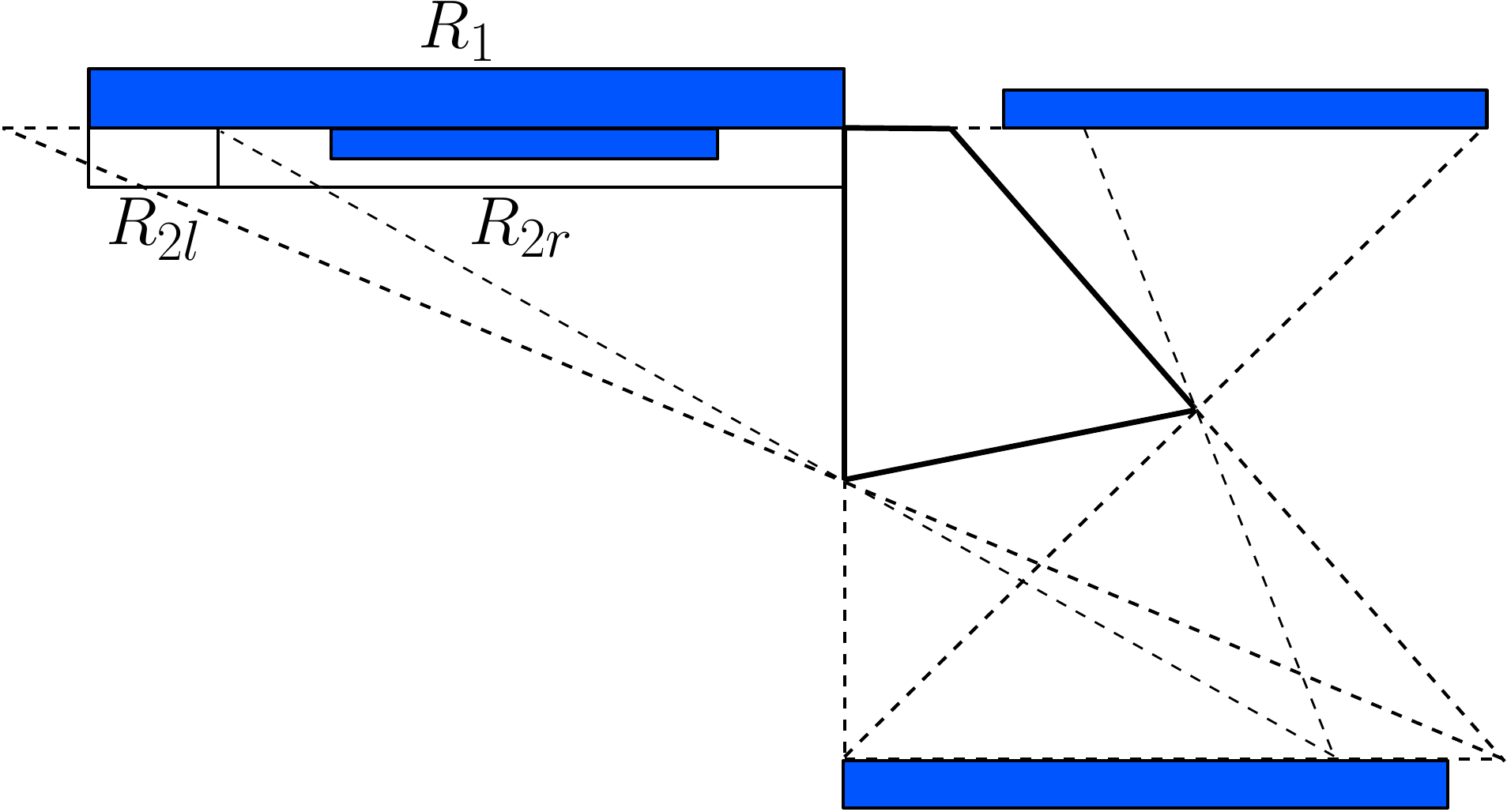}
	\includegraphics[scale=0.3]{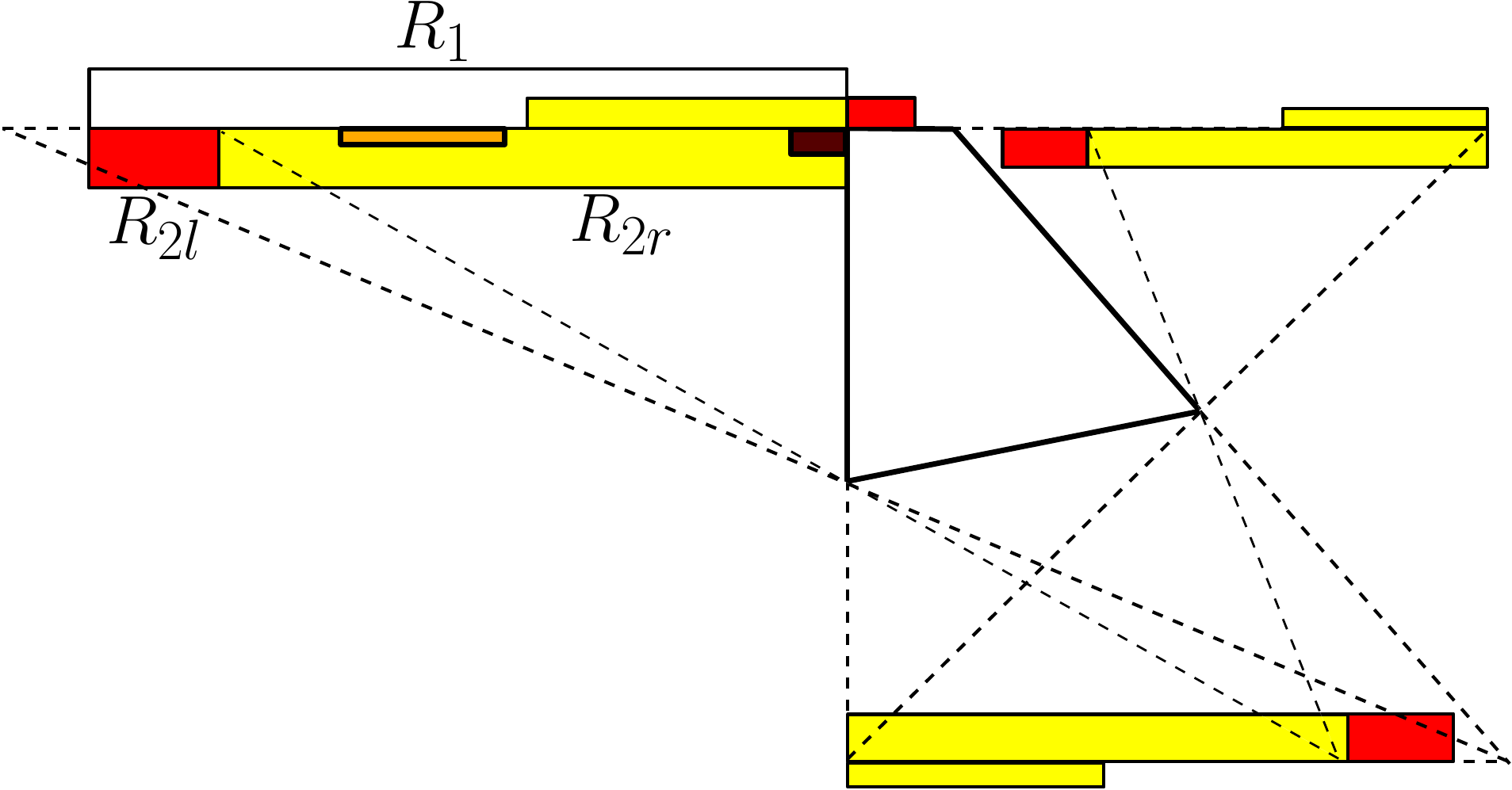}
	\caption{\label{fig:Iterates}Forward iterates of $R_{1}$; $T^{3}R_{1}\subset R_{2}$
		for $a=0.3$ (left) and $a=0.6$ (right)}
\end{figure}

\begin{figure}[h]
	\centering
		\includegraphics[scale=0.3]{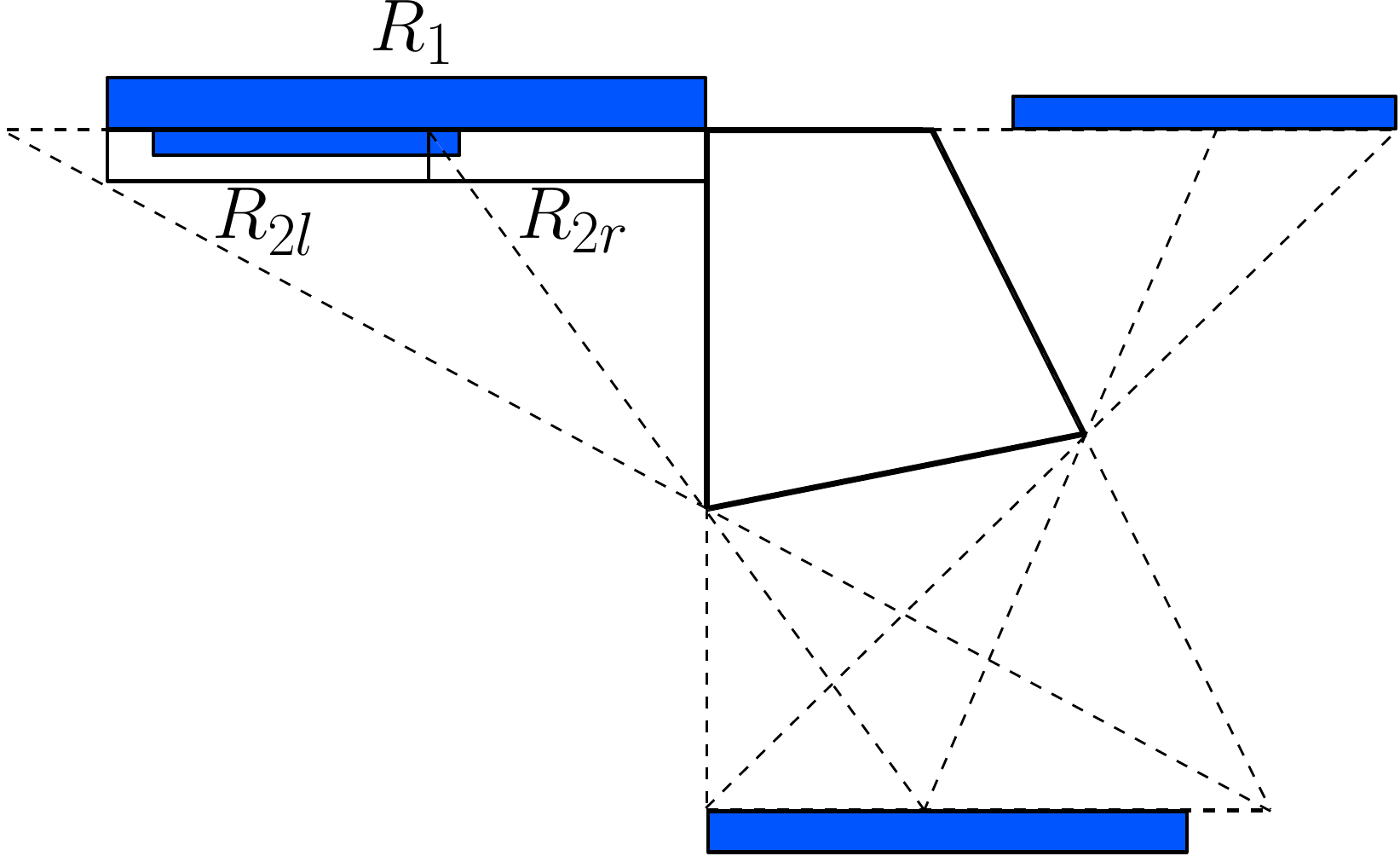}
		\includegraphics[scale=0.3]{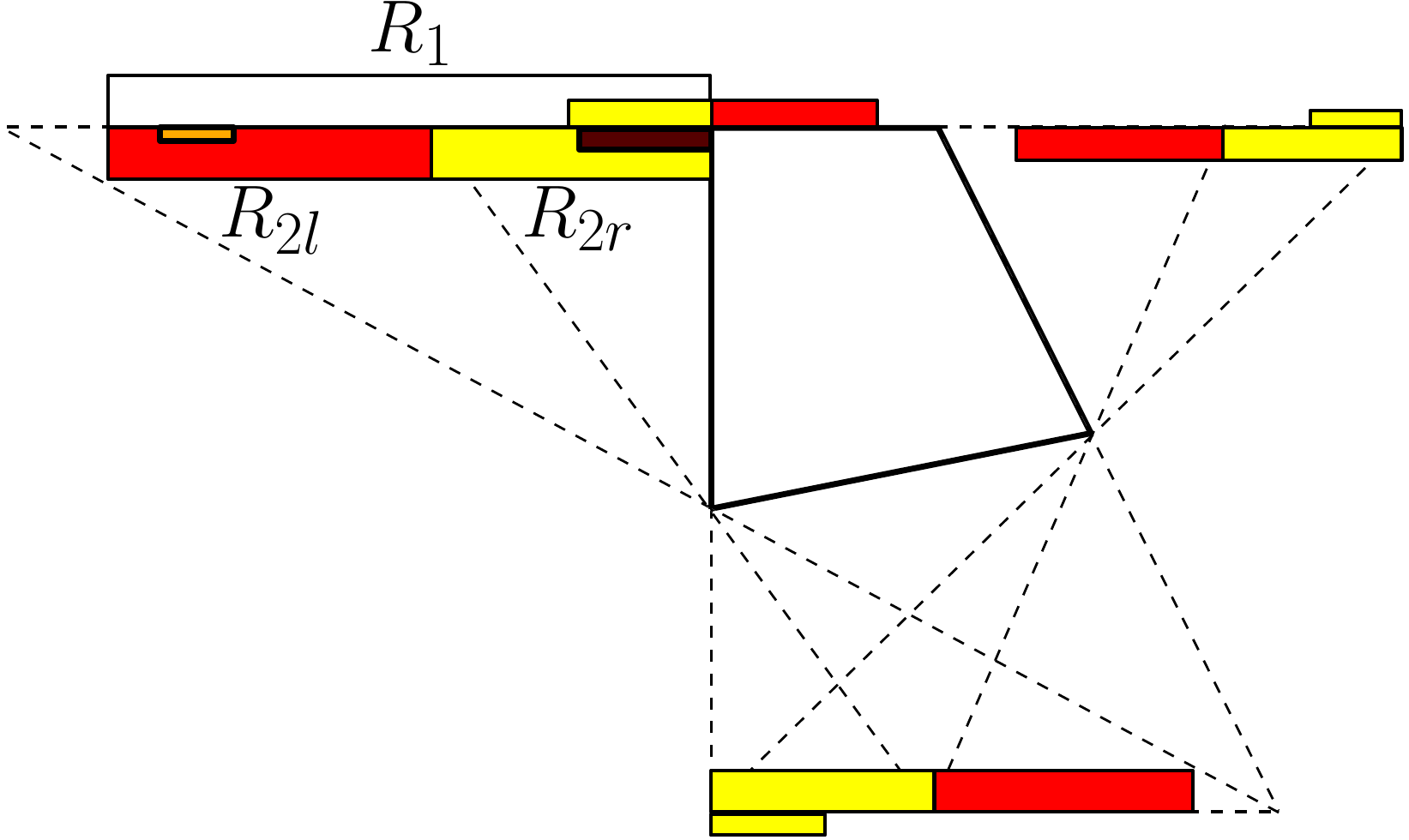}
	
	\caption{\label{fig:Iterates-1}Forward iterates of $R_{2}$; $T^{4}R_{2l}\subset R_{2r}$
		and $T^{6}R_{2r}\subset R_{2}$ for for $a=0.3$ (left) and $a=0.6$
		(right)}
\end{figure}

\begin{definition}[One-dimensional Systems]
	\label{def:1-dimensional_maps} Let $I_{\ep}$ be the interval obtained
	by the intersection $R_{1}\cap R_{2}$. Given a point $p=(p_{1},1)$
	on the interval $I_{\ep}$, pick $\delta>0$ small enough that $(p_{1},1+\delta)\in R_{1}$.
	Vertically project $\hat{f}(p_{1},1+\delta)$ onto $I_{\ep}$ and
	define it to be $f(p)$. This map $f:I_{\ep}\rightarrow I_{\ep}$
	is well-defined independent on the choice of $\delta$. Note that
	if we have two different $\ep_{1}>\ep_{2}$, the map $f$ coincides
	on the intersection $I_{\ep_{1}}\cap I_{\ep_{2}}=I_{\ep_{1}}$. Therefore
	we take the limit $\ep\rightarrow0^{+}$ and then $f$ is well-defined
	on the segment $\overline{EA}$. Similarly, using $\hat{g}$ we define
	$g$ on the segment $\overline{EA}$. 
	
	That is, the map $f$ ($g$, resp.) describes the dynamics of the
	upper rectangle $R_{1}$ (the lower rectangle $R_{2}$, resp.). \end{definition}
\begin{lemma}[Reduction to One-dimensional Systems]
	\label{lem:limit_set_reduction}Let $p\in X\cap(R_{1}\cup R_{2})$.
	Then $\omega_{T}(p)\cap(R_{1}\cup R_{2})=\omega_{T}(p)\cap\overline{EA}=\omega_{f}(p')\cup\omega_{g}(p'')$,
	where $p'$ and $p''$ are vertical projections of some forward $T$-iterates
	of $p$ which lies on $R_{1}$ and $R_{2}$, respectively. \end{lemma}
\begin{proof}
	Let $x\in\omega_{T}(p)\cap\overline{EA}$. We have a sequence $T^{n_{1}}p,T^{n_{2}}p,...$
	on $R_{1}\cup R_{2}$ which converges to $x$. Then infinitely many
	of them lie on $R_{1}$ and hence we get a subsequence which we can
	as well write in the form $\hat{f}^{m_{1}}p,\hat{f}^{m_{2}}p,...$.
	That is, $x\in\omega_{\hat{f}}(p)$. We have $f\circ\mathrm{Proj}=\mathrm{Proj}\circ\hat{f}$
	where $\mathrm{Proj}$ is the projection map from $R_{1}\cup R_{2}$
	onto $\overline{EA}$. Therefore, $x\in\omega_{f}(p')$. 
	
	Let $x\in\omega_{f}(p')$. By our assumption, there is a point $p\in X\cap(R_{1}\cup R_{2})$
	that $p'$ is the projection of some $T^{n}p\in R_{1}$. From the
	sequence of points $f^{n_{1}}(p'),f^{n_{2}}(p'),...$ that converges
	to $x$, we get a sequence of points $\hat{f}^{m_{1}}p,\hat{f}^{m_{2}}p,$...
	that also converges to $x$. Since $\hat{f}$ is just some iterates
	of $T$, we get $x\in\omega_{T_{\l}}(p)\cap\overline{EA}$.
	
	Notice that this result still holds when we view $f$ and $g$ as
	discontinuous circle maps. This identification may create new periodic
	orbits which cannot be obtained from the dynamics of $T$, but such
	a periodic orbit, even when it exists, will not attract any regular
	points as $\omega$-limit sets remain the same.
\end{proof}

Now we prove an important lemma which shows that the 1-dimensional
systems $g$ and $f$ have well-defined rotation numbers.
\begin{lemma}[Rotation Number]
	\label{lem:rotation number} The rotation number of $g$ and $f$
	are well-defined mod 1. Moreover, these rotation numbers vary continuously
	in the parameter $a$. \end{lemma}
\begin{proof}
	We apply an orientation preserving affine map from the interval $\overline{EA}$
	onto $[0,1]$ so that $g$ and $f$ are now maps $[0,1]\rightarrow[0,1]$.
	Then it is a straightforward computation to show that 
	\begin{equation}
	g(x)=\begin{cases}
	\l^{4}(x-1+h/l)+1 & 0<x<1-h/l\\
	\l^{6}(x-1)+(1-y/l) & 1-h/l<x\leq1
	\end{cases}\label{eq:first_return}
	\end{equation}
	where
	\begin{align}
		l & =  \mu+(1-a)\mu^{2}=|\overline{EA}|\\
		h & =  l-a\mu^{3}=|\overline{IA}|\\
		y & =  \l(1+\l)(1+a\l^{2}-\l^{3})
	\end{align}
	are positive constants. For the rotation number of $g$ to be well-defined,
	we only need to check that it is injective (as it is strictly increasing
	on each interval of continuity), which is equivalent to showing that
	two sets $T^{4}R_{2l}$ and $T^{6}R_{2r}$ are disjoint subsets of
	$R_{2}$ (see Figure \ref{fig:Iterates-1} and also Figure \ref{fig:Graphs}
	for visual verifications of this fact). This is again equivalent to
	checking $\lim_{x\rightarrow0^{+}}g(x)>g(1)$ which is proved by a
	simple computation. We then define a lift 
	\[
	G(x)=\begin{cases}
	g(x) & 0<x\leq1-h/l\\
	g(x)+1 & 1-h/l<x\leq1
	\end{cases}
	\]
	on $(0,1]$ and extend to $\mathbb{R}$ by $G(x+1)=G(x)+1$.
	
	Next we argue continuity of the rotation number. Given five parameters
	$\l_{1}$, $\l_{2}$, $c_{1}$, $c_{2}$, and $t$, which all lie
	in $(0,1)$, we can associate the following function:
	\[
	H(x)=\begin{cases}
	\l_{1}x+c_{1} & 0<x\leq t\\
	\l_{2}x+c_{2}+1 & t<x\leq1
	\end{cases}.
	\]
	For each $0.3\leq a\leq0.6$, $G$ takes above form for appropriate
	parameters. If we fix four parameters and vary the remaining one continuously,
	(the closure of) the filled-graph of $H$ varies continuously in the
	Hausdorff topology. It is clear that varying $a$ continuously moves
	five parameters continuously, which implies that the rotation number varies continuously as well.
	
	Similarly, $f$ is a piecewise linear contraction with slopes $\l^{10}$
	and $\l^{6}$. A similar sequence of computations shows that $f$
	is injective and that the rotation number $\rho(f)$ is well-defined
	as well.
	\end{proof}

\begin{theorem}
	\label{thm:two_cases} For each $0.3\leq a\leq0.6$, when $\rho(g)$
	is rational, there is a unique attracting $T$-periodic orbit such
	that every orbit starting from $X\cap(R_{1}\cup R_{2})$ is asymptotic
	to, and when $\rho(g)$ is irrational, every orbit starting from $X\cap(R_{1}\cup R_{2})$
	is asymptotic to a unique invariant Cantor set.\end{theorem}
\begin{proof}
	We have seen that any point in $R_{1}$ has an iterate in $R_{2}$
	and vice versa. Hence we know that a $f$-periodic orbit will induce
	a $g$-periodic orbit and vice versa. Therefore, $\rho(f)$ is rational
	if and only if $\rho(g)$ is rational. 
	
	Now assume that $\rho(g)$ is rational. Then the limit set of any
	point for $f$ and $g$ is finite, so by Lemma \ref{lem:limit_set_reduction},
	the limit set $\omega_{T}(p)$ for $p\in X\cap(R_{1}\cup R_{2})$
	is finite as well; this gives an attracting periodic orbit for $T$.
	Now given a periodic orbit for $T$ intersecting $R_{1}\cup R_{2}$.
	By restriction, we obtain a periodic orbit of $g$. Now we claim that
	there cannot be more than two periodic orbits. This is geometrically
	clear in the case when $\rho(g)=0$ (see Figure \ref{fig:Graphs}).
	So assume that the denominator of $\rho(g)$ is $r\geq2$. The map
	$g$ is a piecewise contraction on two intervals $I_{1}=[0,1-h/l]$
	and $I_{2}=[1-h/l,1]$, and since $g(I_{1})\subset I_{2}$ (which
	is equivalent to $T^{3}R_{1}\subset R_{2}$ in our original system),
	we see that $g^{2}=g\circ g$ is a piecewise contraction on at most
	three intervals. In the same way, $g^{r}$ is a piecewise contraction
	on at most $2r-1$ intervals ($r\geq2$). We know from the rotation
	theory that each periodic orbit of $g$ has (minimal) period $r$.
	Hence for two or more periodic orbits for $g$ to exist, $g^{r}$
	must have at least $2r$ fixed points. However, on each continuity
	interval for $g^{r}$, its graph can intersect the diagonal $y=x$
	at most once since the slope is less than $1$. 
	
	Next we assume that $\rho(g)$ is irrational. The $\omega$-limit
	sets of $g$ and $f$ cannot be the entire circle since they should
	be invariant under a contraction. Then $\omega_{f}(p')$ and $\omega_{g}(p'')$
	are both homeomorphic to Cantor sets for $p\in X\cap(R_{1}\cup R_{2})$,
	and $\omega_{T}(p)\cap\overline{EA}$ is homeomorphic to a Cantor
	set, being a union of two Cantor sets. Now, $\omega_{T}(p)$ is obtained
	by taking the union of finitely many iterates of $\omega_{T}(p)\cap\overline{EA}$,
	so it is a Cantor set as well. 
\end{proof}

The main result of \cite{Nogueira2012} says that $g$ can have at
most two attracting periodic orbits; in this case we have reduced
the number to 1 due to the extra condition $g(I_{1})\subset I_{2}$.
We are not claiming that this is the only periodic orbit; but it is
the only attracting one. We also note that if the rotation number
of $g$ is rational and has denominator $r$, then the corresponding
attracting $T$-periodic orbit has period greater than $3r$. In particular,
we have proved:

\begin{corollary}
	\label{cor:In-the-parameter}In the parameter range $0.3\leq a\leq0.6$,
	there are attracting periodic orbits for $T$ of arbitrarily high
	period intersecting $R_{1}\cup R_{2}$.\end{corollary}
\begin{theorem}
	\label{thm:cantorset}For uncountably many choices of $a$, there
	exists a Cantor set which attracts all points in $R_{1}\cup R_{2}$.\end{theorem}
\begin{proof}
	Let us compute $\rho(g)$ when $a=0.3$ and $a=0.6$. In the former
	case, we have $g(1-h/l)<1-h/l$. Therefore, by the intermediate value
	theorem, there is a fixed point and $\rho(g)=0$. When $a=0.6$, $g(1-h/l)>1-h/l$
	and there are no fixed points which means $\rho(g)\neq0$. (Figure
	\ref{fig:Graphs}) Therefore, there are uncountably many values of
	$a$ between 0.3 and 0.6 such that $\rho(g)$ and $\rho(f)$ are irrational.
\end{proof}

\begin{figure}[h]
	\centering
		\includegraphics[scale=0.3]{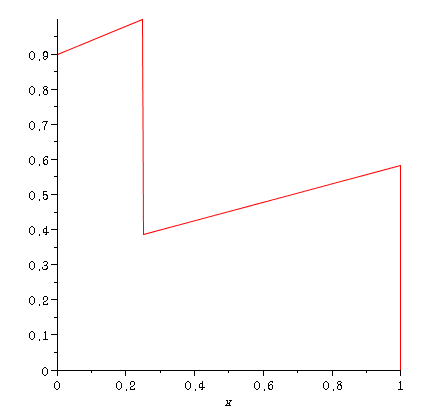}
		\includegraphics[scale=0.3]{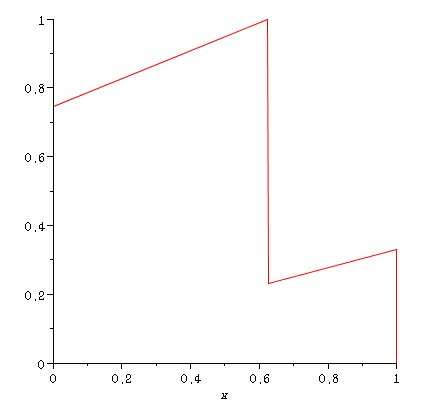}

	\caption{\label{fig:Graphs}Graphs of $g$ for $a=0.3$ and $0.6$ (left, right)}
\end{figure}

\section{Discussions\label{sec:Discussions}}

\subsection{The Triangular Transition}

In this subsection, we will first discuss the behavior of the map
$T$ when $\l$ is close to $0.8$ near the singular line $\overleftrightarrow{AB}$,
and state the results of previous section for more general parameters.

The outer billiards map, either with or without contraction, is invariant
under affine transformations of the plane. That is, if a convex polygon
$P$ is mapped to another polygon $Q$ by an orientation-preserving
affine map, the dynamics of $T$ (either with or without contraction)
outside $P$ and $Q$ are conjugated by a linear map. Up to affine
transformations, any convex quadrilateral is represented by a pair
of reals $(a,b)$ where $0<a$, $b<1$, and $ab<1$. This pair represents
the quadrilateral with vertices $A=(0,1)$, $B=(a,1)$, $C=(1,b)$,
and $D=(0,0)$. Moreover, we can cyclically rename the vertices; the
permutation $(A,B,C,D)$$\rightarrow$$(D,A,B,C)$ correspond to the
transformation $(a,b)\mapsto(1-b,(1-1/a)/(1-b))$ and one can easily
check from this that we can further assume $0<a<1$ and $0<b<1$ without
loss of generality.

Let us say that a periodic orbit is\textit{ triangular} if its period
is 3. If we consider quadrilaterals satisfying $a+b\leq1$, then two
triangular periodic orbits exist in disjoint ranges of $\l$. It can
be proved by a simple computation. Furthermore, let us say that a
periodic orbit is \textit{regular} if every point lies on the set
$X$. Otherwise, a periodic orbit will be called \textit{degenerate}.
While a regular periodic orbit is necessarily locally attracting,
a degenerate periodic orbit of odd period is never locally attracting. 
\begin{proposition}
	\label{1-periodic orbits}If $a+b\leq1$, the quadrilateral $(a,b)$
	has the regular triangular periodic orbit skipping the vertex $A$
	(resp. vertex $B$) precisely in the range $a<\l<1-b$ (resp. $1-b<\l<1$).
\end{proposition}
That is, at the critical value $\l=1-b$, two triangular periodic
orbits both exist but in degenerate forms; each periodic orbit has
two points on the line $\overleftrightarrow{AB}$. Also, we note that
in the range $\l\in(a,1)\backslash\{1-b\}$, some neighborhood of
the singular ray extending $\overline{AB}$ is attracted to the triangular
periodic orbit. However, when $\l=1-b$, both of them are not attracting
(since the period is odd), so it is believable that a new attractor
should appear to compensate for this loss. We have seen in Corollary
\ref{cor:In-the-parameter} and Theorem \ref{thm:cantorset} that
such a new attractor is either a degenerate periodic orbit of high
period or a Cantor set. 

To illustrate this by an example, let us consider the quadrilateral
$P=(0.5,0.2)$. In this specific case, in the range $a=0.5<\l<0.8=1-b$,
the whole domain is asymptotic to the triangular orbit skipping $A$
(left one in Figure \ref{fig:bif}), while in the range $0.8<\l<0.85$,
the whole domain is attracted to the triangular orbit skipping $B$.
(right one in Figure \ref{fig:bif}) At the critical value $\l=0.8$,
a degenerate periodic orbit of period 10 (one of the ten points lies
near the midpoint of the edge $\overline{AB}$) suddenly appears and
attracts the whole domain. In general for different quadrilaterals,
a few other periodic orbits can simultaneously exist near the bifurcation
value $\l=1-b$ but the ones that exist for values of $\l$ smaller
than $1-b$ are qualitatively different from the ones that exist for
$\l$ greater than $1-b$. See \cite{Jeong2013} for more explanations
and pictures regarding this issue.

We state the most general parameter range in which considerations
from the previous section apply without any extra effort.
\begin{theorem}
	Let $(a,b)$ satisfy $0<a,b<1$, $a+b\leq1$, $a<1+\l-\l^{3}-\l^{4}$.
	Then with $P=(a,b)$ and $\l=1-b$, $g$ is likewise well-defined
	piecewise contraction given by the formula \ref{eq:first_return}
	whose rotation number is well-defined, and when $\rho(g)$ is rational,
	at most two periodic orbits exist and every point in $X\cap(R_{1}\cup R_{2})$
	is asymptotic to one of them, and when $\rho(g)$ is irrational, every
	point in $X\cap(R_{1}\cup R_{2})$ is asymptotic to a unique Cantor
	set.
\end{theorem}
We show in Figure \ref{fig:Vertical-axis:-,} how the rotation number
$\rho(g)$ varies in the parameter plane $0<a,b<1$, $a+b\leq1$,
$a<1+\l-\l^{3}-\l^{4}$. Without the inequality $a<1+\l-\l^{3}-\l^{4}$
(which guarantees $T^{3}R_{1}\subset R_{2}$ among other things),
the dynamics is quite different and more complicated. The central
``band'' corresponds to $\rho=0.5$, and the large upper and lower
regions correspond to $\rho=0$. If we fix $b\in(0,0.5)$ and then
the graph of the function $a\mapsto\rho(g_{a,b})$ (where $g_{a,b}$
is the map $g$ with the polygon $P=(a,b)$) is a devil's staircase.

Next, let us take a look at cases where the rotation numbers are rational. In Figure \ref{fig:bif-1}, the bifurcation attracting periodic orbits
are drawn for $P=(a,b=0.15)$ where $a=$0.30, 0.33, 0.34, 0.35, 0.40,
0.50 in clockwise order starting from the top left picture. They correspond
to rotation numbers 1, 6/7, 4/5, 3/4, 2/3, 1/2, respectively. One
can easily see that the period increases in proportion to the denominator
of the rotation number.

\begin{figure}[h]
\centering
		\includegraphics[scale=0.15]{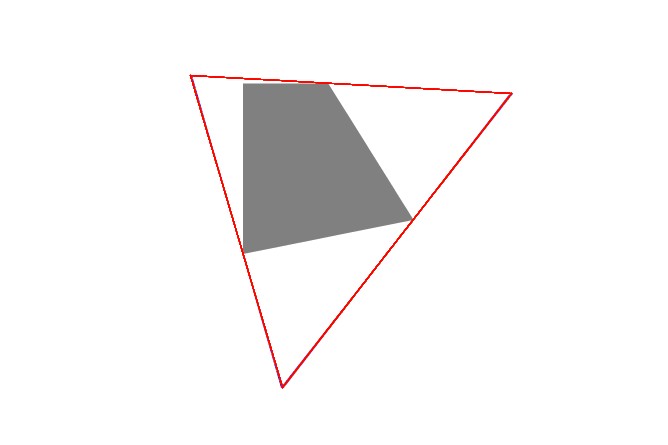}
		\includegraphics[scale=0.15]{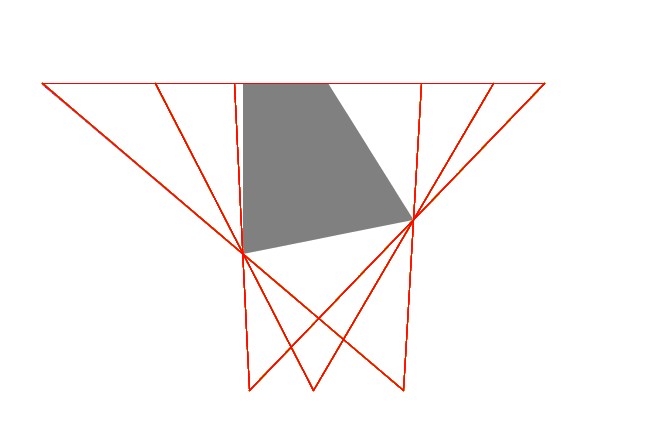}
		\includegraphics[scale=0.15]{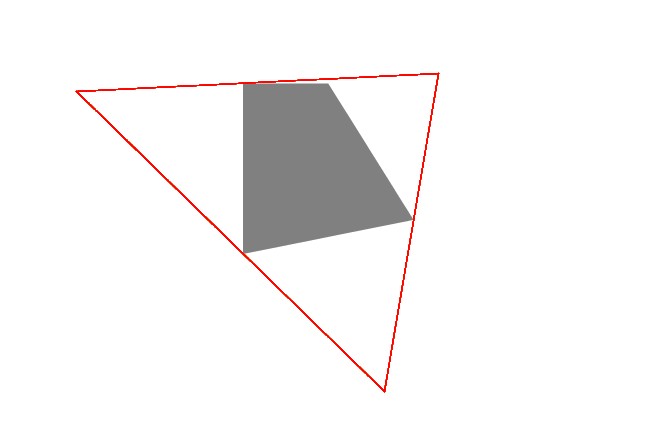}

	\caption{\label{fig:bif}Attracting periodic orbits for $P=(0.5,0.2)$ when
		$\l=0.75,0.8,0.85$ (from left to right)}
\end{figure}

\begin{figure}[h]
		\centering
		\includegraphics[scale=0.4]{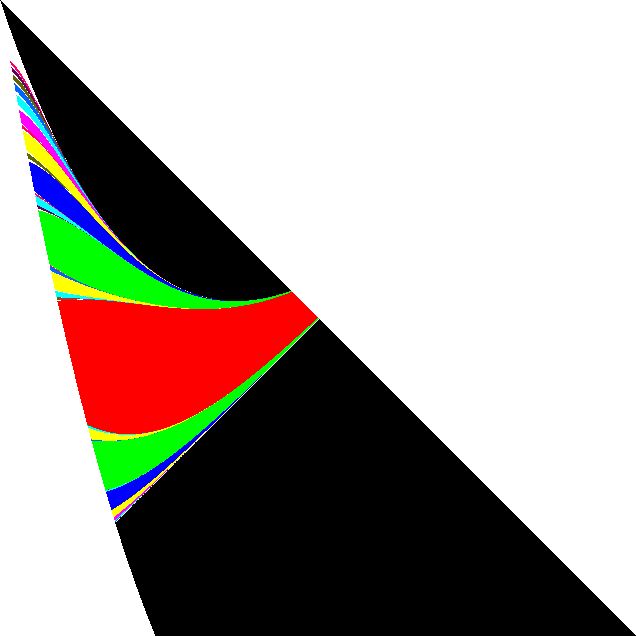}

		\caption{\label{fig:Vertical-axis:-,} Rotation number $\rho(g)$ as a function
		of $a$ and $b$; the top left corner is $a=b=0$, and the horizontal
		axis is $b$.}
\end{figure}

\begin{figure}[h]
\centering
		\includegraphics[scale=0.13]{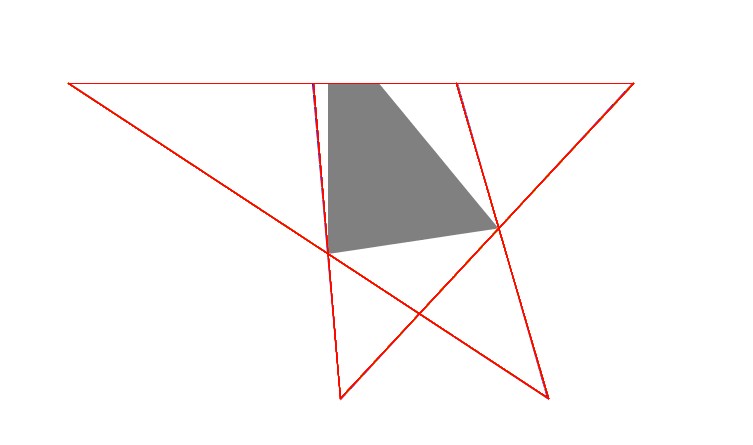}
		\includegraphics[scale=0.13]{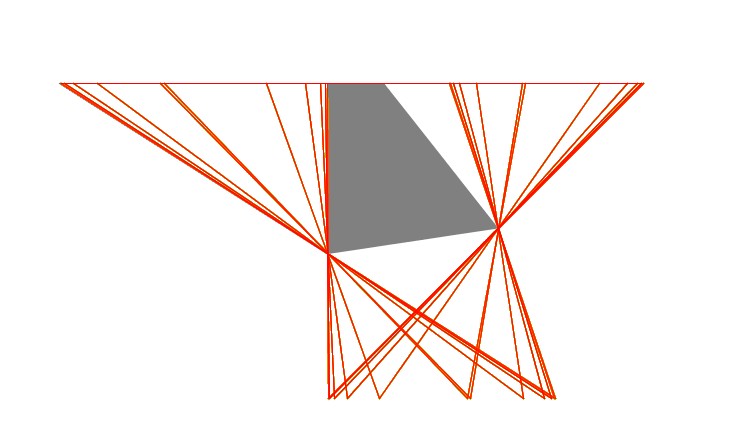}
		\includegraphics[scale=0.13]{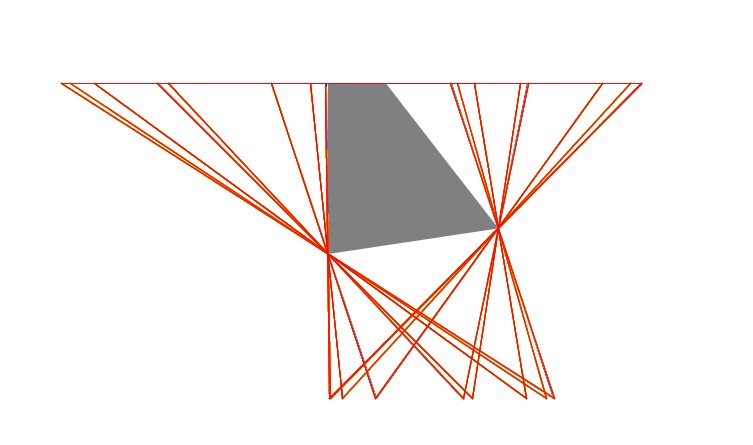}
	
\centering
		\includegraphics[scale=0.13]{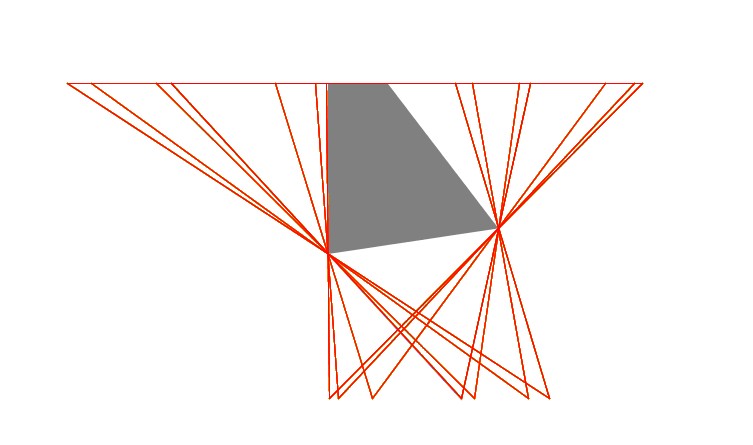}
		\includegraphics[scale=0.13]{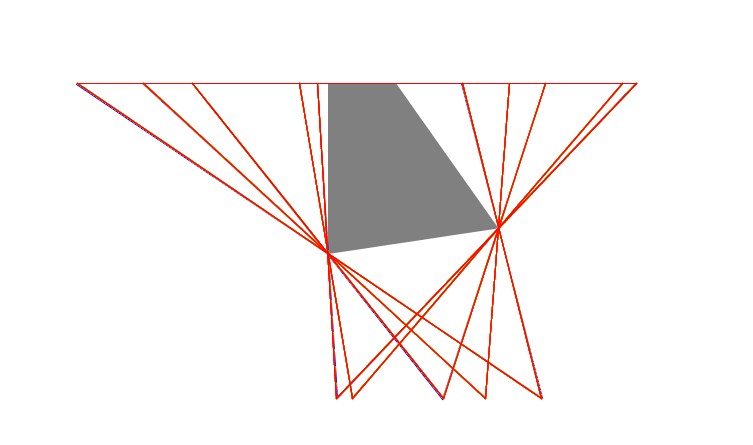}
		\includegraphics[scale=0.13]{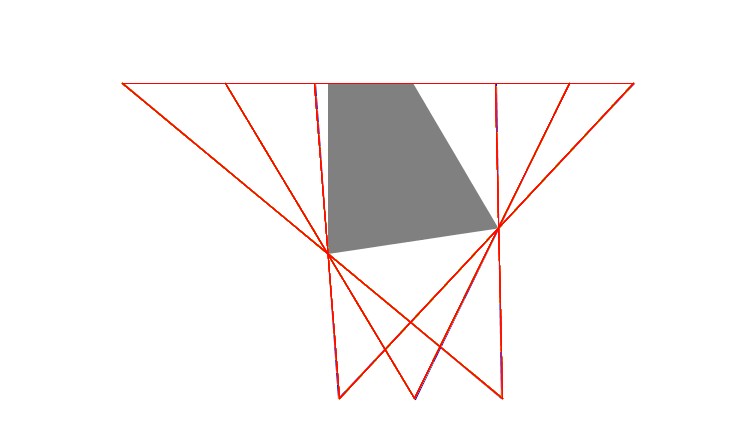}
	
	\caption{\label{fig:bif-1} Attracting periodic orbits for various rational	rotation numbers}
\end{figure}

\subsection{\label{sub:Global-Uniqueness}Global Uniqueness and Unique Ergodicity}

One might suspect that the set of points in $\mathbb{R}^{2}\backslash P$
that are asymptotic to the attractor constructed above is rather small.
Indeed, experiment suggests that for most pairs of $(a,b)$ considered
in the previous subsection, every regular point $p\in\mathbb{R}^{2}\backslash P$
has a forward iterate in $R_{1}\cup R_{2}$. Or equivalently, the
inverse iterates of $R_{1}\cup R_{2}$ covers the whole domain $\mathbb{R}^{2}\backslash P$.
While this seems to hold in particular for the 1-parameter family
of quadrilaterals considered the previous section, we will only prove
this statement for the parameters $0.4\leq a\leq0.5$ and $b=0.4$
(hence $\l=0.6$). Later it will become clear why this statement is
easier to verify for small $\l$. 

For each pair $(a,b)$, we denote the unique attractor for points
in $X\cap(R_{1}\cup R_{2})$ by $K_{(a,b)}$.
\begin{theorem}
	In the parameter range $0.4\leq a\leq0.5$ and $b=0.4$ ($\l=0.6$),
	the entire domain is asymptotic to $K_{(a,b)}$. In particular, there
	exists values of $(a,b)$ for which the entire domain is asymptotic
	to a single Cantor set.
	\end{theorem}

\begin{proof}
	We first note that when $a=0.4$, $\rho(g)=0$ and when $a=0.5$,
	$\rho(g)=1/3$. Hence, there exist values of $a$ between $0.4\leq a\leq0.5$
	such that $\rho(g)$ is irrational. Now the proof consists of showing
	three statements:
	\begin{enumerate}
		\item First, we construct a forward-invariant ball $\mathcal{B}$ such that
		every point in $\mathbb{R}^{2}\backslash P$ has a forward iterate
		in $\mathcal{B}$. 
		\item Then we construct a pentagonal region $Z_{0}$ which for all $\ep>0$
		(thickness of the rectangles $R_{1}$ and $R_{2}$) satisfies $T^{N(\ep)}Z_{0}\subset R_{1}\cup R_{2}$
		for some large $N(\ep)>0$. 
		\item Finally, we show that $\mathcal{B}\subset Z_{0}\cup T^{-1}Z_{0}\cup T^{-2}Z_{0}\cup T^{-3}Z_{0}$. 
	\end{enumerate}
	Then from 1, 2, and 3, we are done; any point in $\mathbb{R}^{2}\backslash P$
	has a forward iterate in $\mathcal{B}$ (by 1), and then it has a
	forward iterate in $Z_{0}$ (by 3), which has a forward iterate in
	$R_{1}\cup R_{2}$ (by 2). These statement will be proved in Lemmas
	\ref{lem:BALL}, \ref{lem:forward_iterations}, and \ref{lem:backward_iterates},
	respectively.\end{proof}

\begin{lemma}
	\label{lem:BALL} Take $\ep'>0$ arbitrary. Consider the max metric
	on the plane: $||(x,y)||=\max\{|x|,|y|\}$ and we define $\mathcal{B}$
	be the ball of radius $2+\ep'$ in this metric with center $(0.5,0.5)$.
	Then for $0.4\leq a\leq0.5$ and $b=0.4$, it is forward-invariant
	and every point in $\mathbb{R}^{2}\backslash P$ has a forward iterate
	in $\mathcal{B}$. \end{lemma}
\begin{proof}
	This follows from a general fact; let $||\cdot||$ be any norm on
	the plane, and let $\{v_{1},...,v_{n}\}$ be the set of vertices of
	$P$. Then for a point $w\in\mathbb{R}^{2}\backslash P$, let us assume
	that 
	\[
	||w||=\frac{1+\lambda}{1-\lambda}\max_{i}||v_{i}||+\delta
	\]
	where $\delta>0$. Since $Tw=(1+\l)v-\l w$ for some vertex $v$ of
	$P$, we have
	
	\[
	||Tw||\leq(1+\l)\max||v_{i}||+\l\frac{1+\l}{1-\l}\max||v_{i}||+\l\delta=\frac{1+\lambda}{1-\lambda}\max_{i}||v_{i}||+\l\delta
	\]
	and hence 
	\[
	\limsup_{n\rightarrow\infty}||T^{n}w||\leq\frac{1+\lambda}{1-\lambda}\max_{i}||v_{i}||.
	\]
	Moreover, when we have 
	\[
	||w||\leq\frac{1+\lambda}{1-\lambda}\max_{i}||v_{i}||,
	\]
	we deduce 
	\[
	||Tw||\leq\frac{1+\lambda}{1-\lambda}\max_{i}||v_{i}||.
	\]
	Now, in our circumstances we have $\l=0.6$ and $||v_{i}||=0.5$ for
	all $i$, with respect to the center $(0.5,0.5)$. Hence any ball
	of radius strictly greater than $2$ centered at $(0.5,0.5)$ will
	do the job. 
\end{proof}

We will now present the region $Z_{0}$. For this, it is necessary
to construct some extra points: (see Figure \ref{fig:points_and_regions})
\begin{itemize}
	\item $D^{1}=(0,\frac{8}{3})$: the point on $\overrightarrow{DA}$ such
	that $|D^{1}A|:|AD|=1/\l$. 
	\item $C^{1}=(-\frac{5}{3},-\frac{2}{3})$: the point on $\overrightarrow{CD}$
	such that $|C^{1}D|:|DC|=1/\l$. 
	\item $D^{2}=(\frac{8}{3},-\frac{152}{45})$: the point on $\overrightarrow{D^{1}C}$
	such that $|D^{2}C|:|CD^{1}|=1/\l$. 
	\item $L=(\frac{8}{3},-\frac{2a+3}{5(1-a)})$: the intersection of $\overleftrightarrow{BC}$
	with the vertical line through $D^{2}$. 
	\item $L^{1}=(-\frac{40}{9},\frac{2a+3}{3(1-a)})$: the point on $\overrightarrow{LD}$
	such that $|L^{1}D|:|DL|=1/\l$. 
	\item $C^{2}=(\frac{25}{9},\frac{34}{9})$: the point on $\overrightarrow{C^{1}A}$
	such that $|C^{2}A|:|AC^{1}|=1/\l$.
	\item $E=(\frac{5}{9}(8-5a),1)$: the intersection of the line $\overleftrightarrow{AB}$
	with the line $\overleftrightarrow{L^{1}C^{1}}$. 
	\item $E^{1}=(\frac{25}{27}(8-5a),1)$: the point on $\overrightarrow{EA}$
	such that $|E^{1}A|:|AE|=1/\l$.
	\item $M=(\frac{5}{9}\frac{34a-49}{2a-5},\frac{2}{9}\frac{34a-49}{2a-5})$:
	the intersection of $\overleftrightarrow{C^{2}J^{1}}$ with $\overleftrightarrow{CD}$. 
	\item $M^{1}=(-\frac{25}{27}\frac{34a-49}{2a-5}+\frac{8}{3},-\frac{10}{27}\frac{34a-49}{2a-5}+\frac{16}{15})$:
	the point on $\overrightarrow{MC}$ such that $|M^{1}C|:|CM|=1/\l$.
	\item $A^{1}=(\frac{8a}{3},1)$: the point on $\overrightarrow{AB}$ such
	that $|A^{1}B|:|BA|=1/\l$. Here it is important that the $x$-coordinate
	of $A^{1}$ is greater than the $x$-coordinate of $C$. 
	\item $A^{2}=(-\frac{40a}{9}+\frac{8}{3},-\frac{3}{5})$: the point on $\overrightarrow{A^{1}C}$
	such that $|A^{2}C|:|CA^{1}|=1/\l$.
\end{itemize}
We have four extra points whose coordinates will not be important.
\begin{itemize}
	\item $N_{1}$: the intersection of the vertical line through $A^{1}$ with
	the line $BC$. 
	\item $N_{2}$: the intersection of the vertical line through $A^{2}$ with
	the line $DC$.
	\item $N_{3}$: the intersection of the horizontal line through $A^{2}$
	with the line $DC$.
	\item $N_{4}$: the intersection of the horizontal line through $D^{1}$
	with the vertical line through $L^{1}$. 
\end{itemize}
Now we are ready to define the region $Z_{0}$ as well as other regions
$Z_{1},...,Z_{6}$. We orient the vertices in clockwise order. Again,
see Figure \ref{fig:points_and_regions}.
\begin{itemize}
	\item $Z_{0}$: this is the pentagon with vertices $L^{1},N_{4},D^{1},D$,
	and $C^{1}$. For this to be well-defined, we check that the $y$-coordinate
	of $D^{1}$ is greater than or equal to the $y$-coordinate of $L^{1}$
	in the range $a\in[0.4,0.5]$. Indeed, for $a=0.5$, points $N_{4}$
	and $L^{1}$ coincides and $Z_{0}$ becomes a quadrilateral. 
	\item $Z_{1}$: this is the quadrilateral with vertices $D^{1},C^{2},E^{1}$,
	and $A$. 
	\item $Z_{2}$: this is the quadrilateral with vertices $A^{1},E^{1},L$,
	and $N_{1}$.
	\item $Z_{3}$: this is the triangle with vertices $B,A^{1}$, and $N_{1}$.
	\item $Z_{4}$: this is the right-angle triangle with vertices $N_{2},A^{2},$
	and $N_{3}$.
	\item $Z_{5}:$ this is the pentagon with vertices $N_{3},G,L,D^{2},$ and
	$M^{1}$. 
	\item $Z_{6}$: this is the quadrilateral with vertices $N_{2},C,G,$ and
	$A^{2}$. \end{itemize}
\begin{lemma}
	\label{lem:forward_iterations} For all $\ep>0$, there exists large
	$N>0$ such that $T^{N}Z_{0}\subset R_{1}\cup R_{2}$ where $\ep$
	is the thickness of rectangles $R_{1}$ and $R_{2}$. \end{lemma}
\begin{proof}
	The segment $\overline{EA}$ divides $Z_{0}$ to two regions; denote
	them by $Z_{0}^{+}$ (one above) and $Z_{0}^{-}$ (one below). Note
	that these regions contain the segment $\overline{EA}$, which is
	the domain of our first-return maps $f$ and $g$ analyzed in the
	previous section. Therefore, it is enough to prove that $Z_{0}$ is
	forward-invariant and that points in $Z_{0}$ converge to the line
	segment $\overline{EA}$ under the iteration of the first-return map
	to $Z_{0}$. This proof is parallel to the proof of Lemma \ref{lem:invariance}
	and we will only sketch it. See Figure \ref{fig:forward_iterates}.
	
	Consider the rectangle which has three vertices $N_{4},D^{1},$ and
	$A$. This rectangle contains $Z_{0}^{+}$. Then it is elementary
	to check that $T^{3}$ of this rectangle is again a rectangle contained
	in $Z_{0}^{-}$ with one vertex $A$. The $y$-coordinates of points
	in $Z_{0}^{+}$ relative to $1$ (the $y$-coordinate of the line
	$AB$) have contracted by $\l^{3}$ in $T^{3}(Z_{0}^{+})$. 
	
	Next, note that $T^{3}(Z_{0}^{-})$ naturally splits into two regions,
	say $Z'$ (the one on the left) and $Z''$ (the one on the right),
	by the vertical line $AD$. First, $Z'$ is contained in $Z_{0}^{+}$
	so we are good. Next, $T(Z'')$ is contained in $Z_{0}^{-}$ and we
	note that the $y$-coordinates of the points in $T(Z'')$ relative
	to $1$ have contracted by $\l^{4}$. We are done.
\end{proof}

\begin{figure}[h]
\centering
		\includegraphics[scale=0.4]{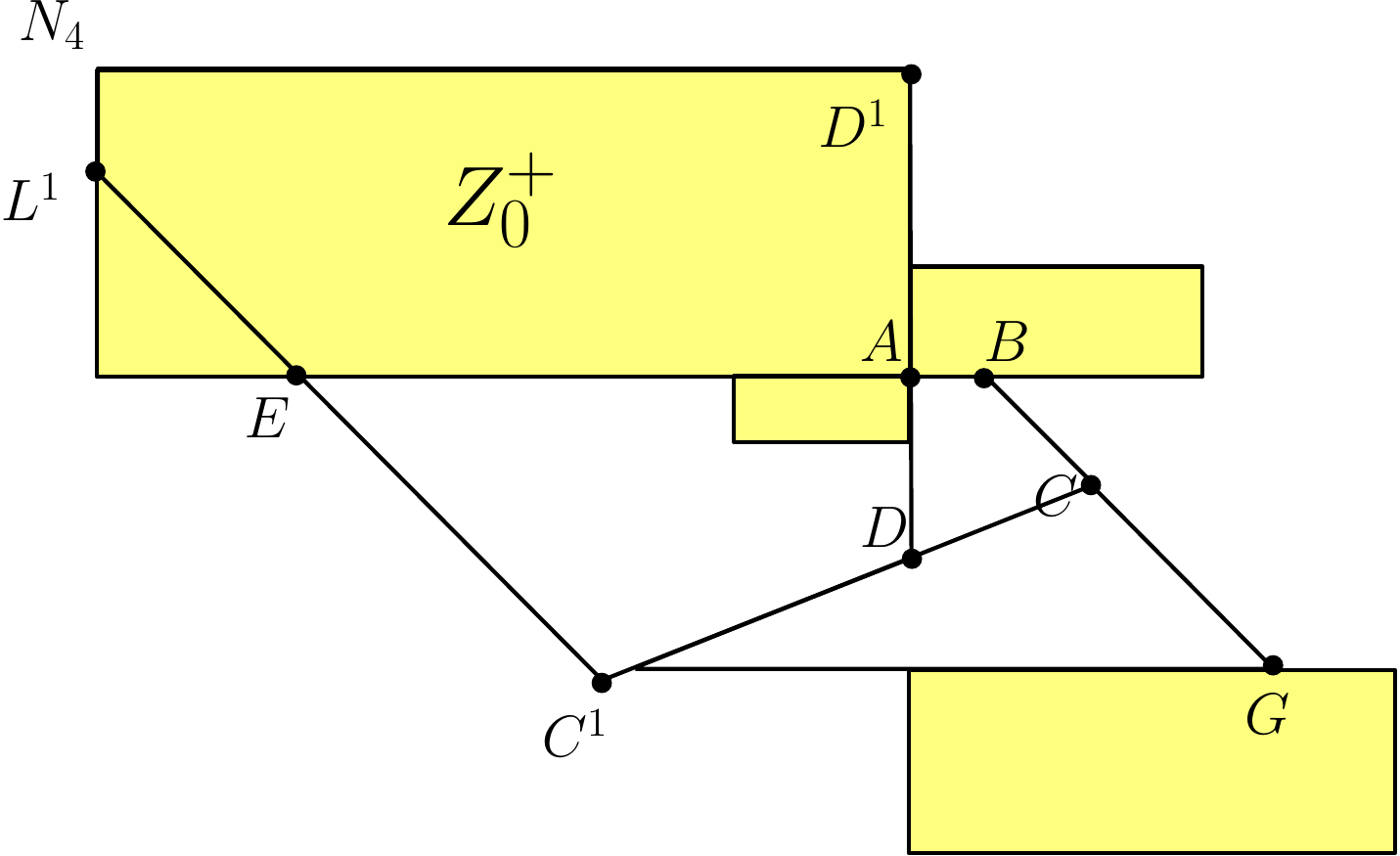}
		\includegraphics[scale=0.4]{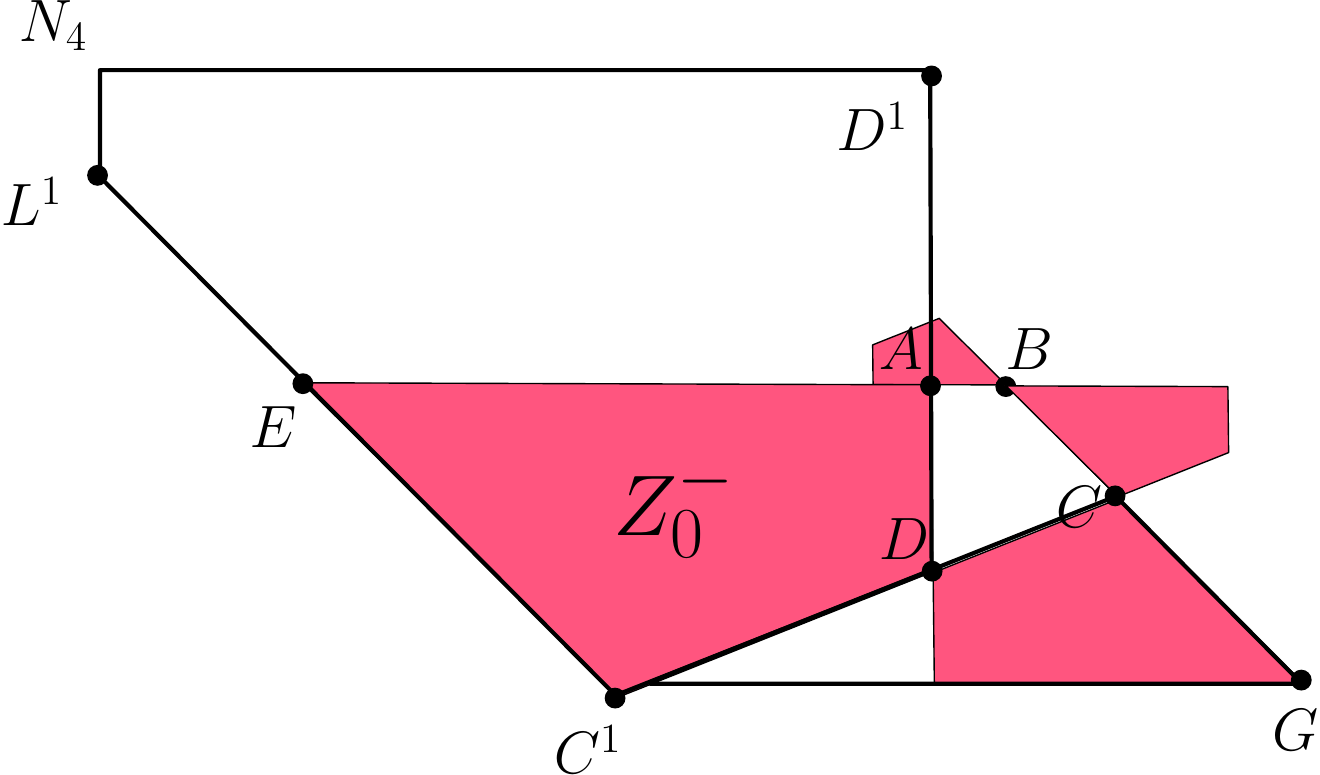}

\caption{\label{fig:forward_iterates}Left: $T^{3}Z_{0}^{+}\subset Z_{0}^{-}$,
	right: $T^{3}Z_{0}^{-}=Z'\cup Z''$ where $Z'\subset Z_{0}^{+}$ and
	$T(Z'')\subset Z_{0}^{-}$. (drawn for $a=0.4$)}
\end{figure}

The following is the most tedious lemma to prove. It will be much harder to prove for larger values of $\l$. 
\begin{lemma}
	\label{lem:backward_iterates} $\mathcal{B}\subset Z_{0}\cup T^{-1}Z_{0}\cup T^{-2}Z_{0}\cup T^{-3}Z_{0}$. \end{lemma}
\begin{proof}
	We prove $\mathcal{B}\subset\cup_{i=0}^{6}Z_{i}$ and $\cup_{i=0}^{6}Z_{i}\subset Z_{0}\cup T^{-1}Z_{0}\cup T^{-2}Z_{0}\cup T^{-3}Z_{0}$.
	
	For the first statement, it is enough to check the following for all
	$a\in[0.4,0.5]$, each of which is a simple computation using the
	expression for coordinates given above:
	\begin{itemize}
		\item the $x$-coordinate of $C^{1}$ is strictly less than $-1.5$.
		\item the $y$-coordinates of $M^{1}$ and $D^{2}$ are both strictly less
		than $-1.5$.
		\item the $x$-coordinate of $L$ is strictly greater than $2.5$. 
		\item the $y$-coordinate of $D^{1}$ is strictly greater than $2.5$. 
		\item finally, the point $(2.5,2.5)$ (upper right corner) lies below the
		line $C^{2}E^{1}$. 
	\end{itemize}
	In Figure \ref{fig:points_and_regions}, the ball of radius $2$ is
	drawn for the case $a=0.4$. 
	
	The second statement is reduced to check all of the following:
	\begin{itemize}
		\item $T(Z_{1})=Z_{0}^{-}$: this is simply how $Z_{1}$ is defined.
		\item $T(Z_{2})\subset Z_{0}^{+}$: enough to check that the vertex $L$,
		after reflecting on $B$, gets inside $Z_{0}^{+}$. 
		\item $T^{2}(Z_{3})\subset Z_{0}^{-}$: this is straightforward.
		\item $T(Z_{4})\subset Z_{2}$: enough to check that the vertex $N_{3}$,
		after reflecting on $C$, gets inside the segment $\overline{A^{1}E^{1}}.$
		Hence $T^{2}(Z_{4})\subset Z_{0}$. 
		\item $T(Z_{5})\subset Z_{1}$: this follows from the definition of $Z_{5}$.
		Hence $T^{2}(Z_{5})\subset Z_{0}$.
		\item $T(Z_{6})\subset Z_{3}$: this is simply how $Z_{6}$ is defined.
		Hence $T^{3}(Z_{6})\subset Z_{0}$. 
	\end{itemize}
	We are done.
\end{proof}

For each pair $(a,b)$ satisfying $0<a,b<1$, $a+b\leq1$, $a<1+\l-\l^{3}-\l^{4}$
$(\l=1-b$), recall that we had a unique attractor $K_{(a,b)}$ for
points in $X\cap(R_{1}\cup R_{2})$. Then we have the following result. 
\begin{theorem}
	For each $(a,b)$, there is a unique Borel $T$-invariant probability
	measure $\mu_{(a,b)}$ supported on the set $K_{(a,b)}$, such that
	for each continuous function $\phi$ on $\mathbb{R}^{2}\backslash P$,
	\[
	\lim_{n\rightarrow\infty}\frac{\sum_{i=0}^{n-1}\phi(T^{i}x)}{n}\rightarrow\int\phi|_{K_{(a,b)}}d\mu_{(a,b)}
	\]
	holds for all $x\in X\cap(R_{1}\cup R_{2})$, and for all $x\in X$
	for parameters in the range $a\in[0.4,0.5]$, $b=0.4$.\end{theorem}
\begin{proof}
	We first prove that $g$ is uniquely ergodic (that is, $g$ has only
	one Borel invariant probability measure supported on the common $\omega$-limit
	set of its regular points). In the case when $\rho(g)$ is rational,
	every orbit is asymptotic to a single periodic orbit and hence the
	atomic probability measure supported on this periodic orbit is the
	unique invariant measure. When $\rho(g)$ is irrational, we repeat
	the unique ergodicity argument for the circle homeomorphism (with
	irrational rotation number) explained in detail in \cite{Katok95}.
	One can still construct a semiconjugacy between $g$ and the irrational
	rotation by angle $\rho(g)$. 
	
	In the same way, $f$ is uniquely ergodic as well, and we obtain an
	invariant probability measure of $T$ by a linear combination of invariant
	measures for $f$ and $g$ (viewed as measures on the segment $\overline{EA}$)
	and their pushforwards by $T$. Moreover, given an invariant measure
	for $T$ supported on $K_{a}$, we obtain invariant measures for $f$
	and $g$. Therefore, $T$ is uniquely ergodic on the union of $X\cap(R_{1}\cup R_{2})$
	and its forward iterates. 
	
	Now the statement regarding the Birkhoff average holds for $f$ and
	$g$ and therefore it holds for $T$ as well.
\end{proof}

\section*{Acknowledgements}

A large part of this work was done while the author was participating
in the ICERM undergraduate research program in 2012. 

We thank our advisors Prof. Tabachnikov, Prof. Hooper, Tarik Aougab,
and Diana Davis for their support and guidance. Also we thank Julienne
Lachance and Francisc Bozgan for their enthusiasm in this project
and for various helpful discussions. The computer program developed
by Prof. Hooper and J. Lachance played a crucial role in our research
and is responsible for all the results in this article. The coloring
scheme which produced Figure \ref{fig:Vertical-axis:-,} was also
due to Prof. Hooper. Moreover, he was the first to notice the possibility
of finding an attracting Cantor set. 

It is simply impossible to underestimate the amount of help we had
from Prof. Schwartz. Being the very expert on outer billiards, he
gave us valuable comments and insights throughout the research program.
In addition, he had read the drafts multiple times and his feedback
significantly improved the quality of this paper. In particular, he
suggested to use three rectangles in the proof of the main theorem.

Lastly, we express our sincere gratitude to the referee who read the
manuscript very carefully and provided several important comments.

The author is supported by the Samsung Scholarship.

\bibliographystyle{plain}
\bibliography{outer_billiards}

\begin{figure}[h]
\centering
		\includegraphics[scale=0.5]{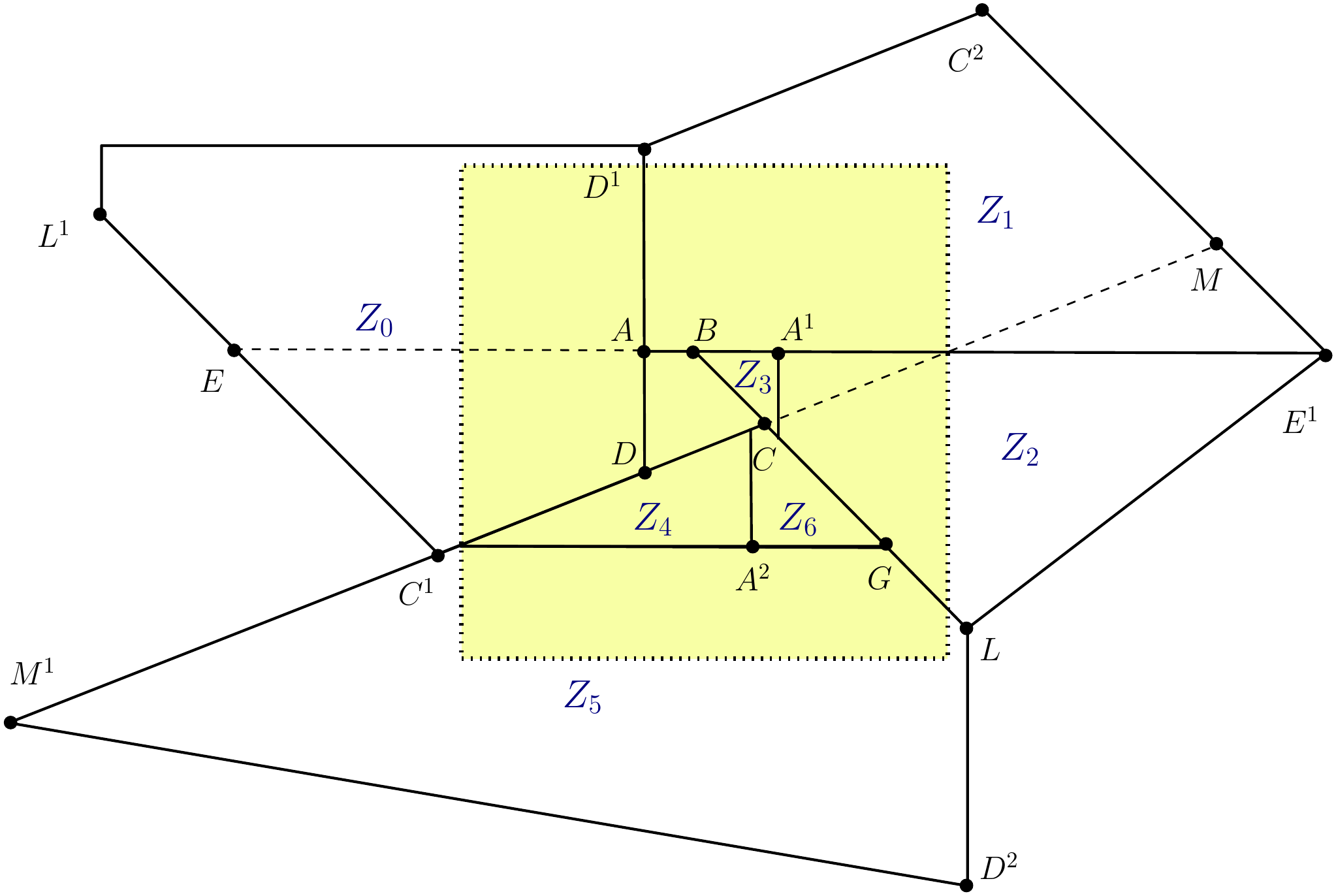}

\centering
		\includegraphics[scale=0.5]{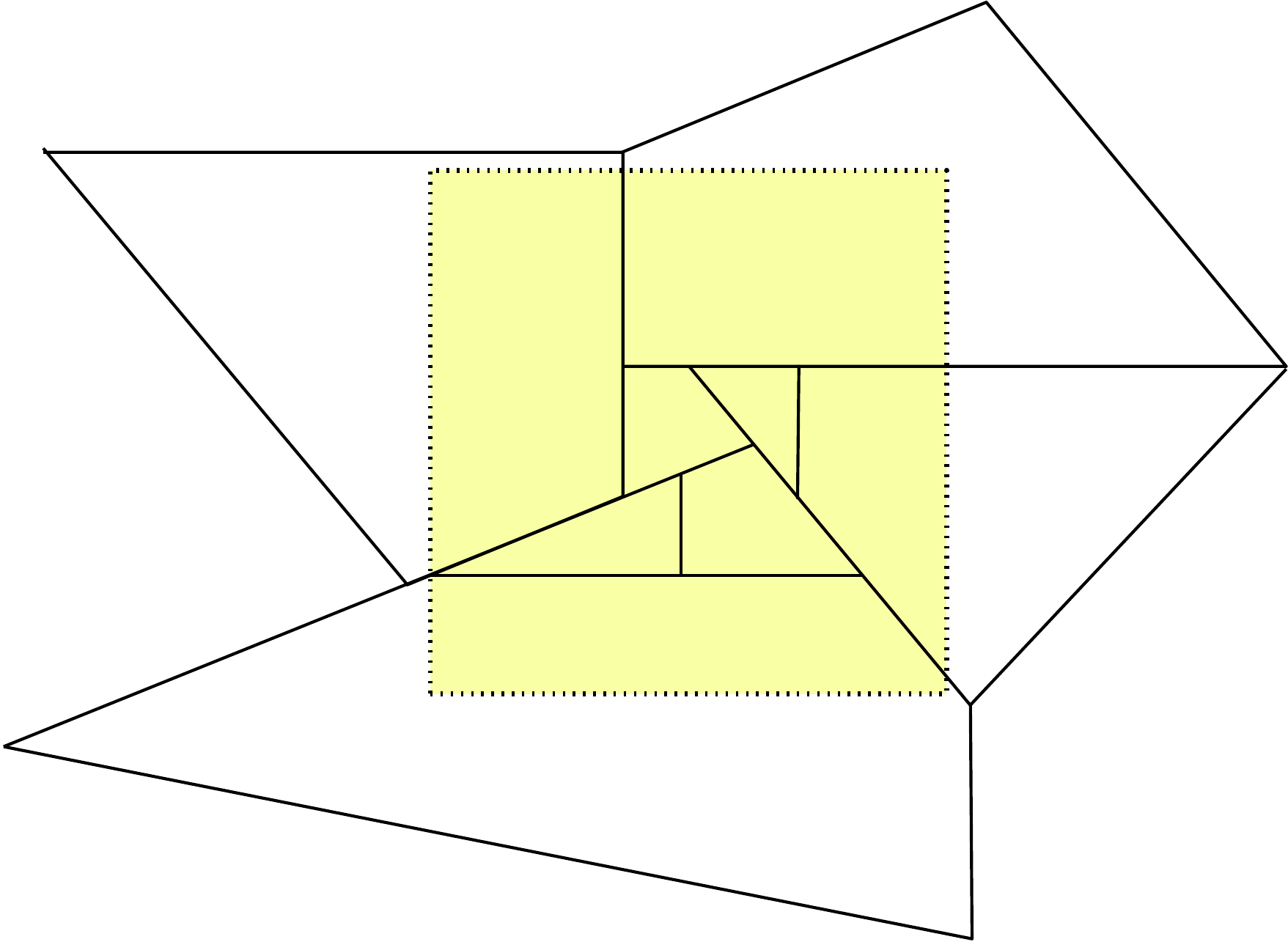}

	\caption{\label{fig:points_and_regions}Extra points and regions, drawn for
		$a=0.4$ (above) and $a=0.5$ (below)}
\end{figure}

\end{document}